\documentclass{interact}

\usepackage{mathtools}
\usepackage{xfrac}
\usepackage{maplestd2e}
\usepackage{amsfonts}
\usepackage{amsmath}
\usepackage{color}

\usepackage[utf8]{inputenc}
\usepackage{csquotes}

\newcommand{\uwk}{\ensuremath{\mathcal{K}}}

\usepackage{listings}
\theoremstyle{thmit} 
\newtheorem{thm}{Theorem}[section]
\newtheorem{lem}[thm]{Lemma}
\newtheorem{cor}[thm]{Corollary}

\newtheorem{rmk}[thm]{Remark}

\DefineParaStyle{Maple Bullet Item}
\DefineParaStyle{Maple Heading 1}
\DefineParaStyle{Maple Warning}
\DefineParaStyle{Maple Heading 4}
\DefineParaStyle{Maple Heading 2}
\DefineParaStyle{Maple Heading 3}
\DefineParaStyle{Maple Dash Item}
\DefineParaStyle{Maple Error}
\DefineParaStyle{Maple Title}
\DefineParaStyle{Maple Text Output}
\DefineParaStyle{Maple Normal}
\DefineCharStyle{Maple 2D Output}
\DefineCharStyle{Maple 2D Input}
\DefineCharStyle{Maple Maple Input}
\DefineCharStyle{Maple 2D Math}
\DefineCharStyle{Maple Hyperlink}

\theoremstyle{thmrm} 

\newtheorem*{oldproof}{Proof}
\renewenvironment{proof}[1][{}]{\begin{oldproof}[#1]}{\qed\end{oldproof}}

\title{Stirling's Original Asymptotic Series from a formula like one of Binet's and its evaluation by sequence acceleration}

\author{
\name{Robert M. Corless\textsuperscript{a} and Leili Rafiee Sevyeri\textsuperscript{a}\thanks{CONTACT Leili Rafiee Sevyeri.
Email: lrafiees@uwo.ca}}
\affil{\textsuperscript{a}Ontario Research Center for Computer Algebra and The School of Mathematical \\ and Statistical Sciences. University of Western Ontario, London, Canada.}
}
\begin{document}
\maketitle
\begin{abstract}
We give an apparently new proof of Stirling's original asymptotic formula for the behavior of $\ln z!$ for large $z$. Stirling's original formula is not the formula widely known as ``Stirling's formula", which was actually due to De Moivre. We also show by experiment that this old formula is quite effective for numerical evaluation of $\ln z!$ over $\mathbb{C}$, when coupled with the sequence acceleration method known as Levin's $u$-transform. As an \textsl{homage} to Stirling, who apparently used inverse symbolic computation to identify the constant term in his formula, we do the same in our proof.
\end{abstract}
\section{Introduction}
Stirling's original formula for the asymptotics of $\ln z!$ has been obscured by the formula popularly known as ``Stirling's formula", namely
\begin{align}\label{1}
\ln z! \sim (z+\frac{1}{2})\ln z -z + \ln \sqrt{2\pi} + z \sum_{n\geq 1}\dfrac{B_{2n}}{2n(2n-1)}\cdot\dfrac{1}{z^{2n}} \\
\sim (z+\frac{1}{2})\ln z -z + \ln \sqrt{2\pi} + \dfrac{1}{12z}-\dfrac{1}{360z^{3}}+\mathcal{O}(\dfrac{1}{z^5})\>,
\end{align}
which was actually found by De Moivre after Stirling had found his (see, \textsl{e.g.},~\cite{bellhouse2011}). Stirling's original formula is
\begin{align}\label{2}
\ln z! \sim Z\ln Z- Z +\ln\sqrt{2\pi}-Z\sum_{n\geq 1}\dfrac{(1-2^{1-2n})B_{2n}}{2n(2n-1)Z^{2n}} \\
\sim Z\ln Z- Z +\ln\sqrt{2\pi} - \dfrac{1}{24Z}+\dfrac{7}{2880Z^3}-\mathcal{O}(\dfrac{1}{Z^5})\>.
\end{align}
where $Z=z+\frac{1}{2}$.\\
\\
As you can see here, the formulae are quite similar. Stirling's original formula in equation~\eqref{2} has been rediscovered several times. Some people call it De Moivre's formula! It seems to have been known to both Gauss and to Hermite (see \textsl{e.g.}~\cite{richard}). There is a discussion in~\cite{tweddle1984} of one such rediscovery in the physics literature; for a particularly ironic case where the rediscoverer claims the formula is ``both simpler and more accurate" than ``Stirling's formula", look at~\cite{spouge1994}. For a thorough exposition of Stirling's actual work see the original, as masterfully translated and annotated by Tweddle~\cite{tweddle1984}. \\
In this present work we give a short proof of equation~\eqref{2}, which we believe to be new, by deriving an apparently new formula that is similar to the following formula of Binet:
    \begin{equation}
    \ln z!=(z+\frac{1}{2})\ln z-z+\ln\sqrt{2\pi}+\int_{t=0}^{\infty}\dfrac{1}{t}  \left( \dfrac{1}{t}-\dfrac{1}{e^t-1}\right) e^{-tz}dt
    \end{equation}
which~\cite{whittaker} claims is valid for $\Re z>0$. We will see later that this is not quite true in the modern context. This classical formula is proved in, for example,~\cite{whittaker} and in~\cite{sasvari}. The new formula is quite similar, again using $Z=z+\frac{1}{2}$:
    \begin{equation}\label{6}
    \ln z!=Z\ln Z-Z +\ln\sqrt{2\pi}-\int_{t=0}^{\infty}\dfrac{1}{t}\left( \dfrac{1}{t}-\dfrac{1}{2\sinh\frac{t}{2}}\right) e^{-tZ}dt,
    \end{equation}
and is valid for $\Re z>-\dfrac{1}{2}$ (again, we will adjust this caveat later). Formula \eqref{6} appears as ``Theorem $2$", without proof, in~\cite{BorweinCorless1999}.\\
In the modern computational world, a new proof of an old mathematical result is rarely of interest for its own sake, but see for instance~\cite{nplus1}. Indeed Stirling's original proof of equation \eqref{2} was algorithmic in nature and, apart from the use of ``recognition" to identify $\sqrt{2\pi}$ and the lack of a ``closed formula"--- \textsl{i.e.} a relationship to other numbers, the Bernoulli numbers--- Stirling's proof was entirely satisfactory. So why record these results? \\
We believe this formula is interesting for the following reasons. First, the rediscovery was identified as such by tracing patterns and citations in Google Scholar, and now there is some hope that the obscurity of the original formula can be lifted\footnote{Of course, there is no hope of changing the popular meaning of the name ``Stirling's formula''.}. Of course the mathematics history literature has it right, owing to the work of Tweddle, but still. Second, Stirling's original proof used what is now called ``Inverse Symbolic Computation," illustrating that a modern experimental technique worth investigation has significant historical roots. As an \textsl{homage} to Stirling we use the same technique in our `new' proof below. Finally, we test Stirling's original formula in a modern computational context by trying a nonlinear sequence acceleration technique, namely Levin's $u$-transform; this gives a surprisingly viable method, comparable in cost (for a given accuracy) to the methods discussed in~\cite{schmelzer}. The separate issue of the complexity of the computation of $\Gamma(1+z)$, $z!$, or $n!$ for $n \in \mathbb{N}$, is not addressed here. See for instance~\cite{peter},~\cite{richard} for entry into that literature. See also~\cite{hare} for the computation of $\Gamma(z)$.\\
    Basic references for $\Gamma$ include the DLMF (chapter $5$), the Dynamic Dictionary, and~\cite{andrews1999}.
%
\section{Notation}
Here we use $z!$ and $\Gamma(z+1)$ interchangeably. As mentioned in~\cite{robgamma} the ``notation wars" and the annoyance of the continual nuisance of shifting by $1$ are amusing but not possible nowadays of resolution. We use $\ln$ for the natural logarithm because it's unambiguous and ingeniously, as pointed out by David Jeffrey offers a free location for a subscript, which we use as follows
\begin{equation}
\ln_k z=\ln z + 2\pi i k \>.
\end{equation}   
The unsubscripted $\ln z$ has range $-\pi < \Im \ln z \leq \pi$, the principal branch in universal usage nowadays in computers. We write $\ln z!$ for $\ln(z!)$; \textsl{i.e.} the factorial has higher precedence. We discuss the function $\ln\Gamma(z)$ in detail below as the analytic continuation of $\ln(z-1)!$. This modern notation is in contrast to Stirling's, where he used $\ell,z$ to mean $\log_{10}z$. The factorial notation $!$ was apparently invented by Christian Kramp in $1808$; the $\Gamma$ notation was invented by Legendre, and although the shift by $1$ as apparently due to Euler himself~\cite{gronau2003gamma}, Legendre gets the blame for that, too.

\section{Divergent asymptotic series}
\bigskip
For a given sequence $\lbrace \phi_i(x) \rbrace$ where the $\phi_i(x)$'s are defined over a domain, one can define a formal series
$\sum_{i= 1}^\infty a_i\phi_i(x)$. The idea of asymptotic series is to define a formal series with special property on the underlying sequence such that its partial sums approximate a given function over the same domain even more closely as $x\rightarrow x_0$.
\\
\\
Assume $R$ is a domain and $\lbrace \phi_i(x) \rbrace$ is a sequence of functions defined over $R$. The sequence $\phi_i(x)$ is called an asymptotic sequence for $x \rightarrow x_0$ in $R$ if for each $i$, $\phi_{i+1}(x)= o(\phi_i(x))$  as $x \rightarrow x_0$. A simple 
example is $\lbrace (x-x_0)^i \rbrace$ for $x \rightarrow x_0$. Recall that $f(x)=o(g(x))$ as $x\rightarrow\infty$ if 
\begin{equation}
\forall c > 0 \;\; \exists N>0 \;\;\; \text{s.t.} \;\;\; |f(x)|<c|g(x)| \;\;\; \text{for} \;\;\; x>N \>.
\end{equation}
or (if $x_0$ is finite)
\begin{equation}
\forall c > 0 \;\; \exists \delta>0 \;\;\; \text{s.t.} \;\;\; |f(x)|<c|g(x)| \;\;\; \text{for} \;\;\; |x-x_0|<\delta \>.
\end{equation}
Now suppose $\lbrace \phi_i(x) \rbrace$ is an asymptotic sequence which is defined over a domain $R$ and $f(x)$ is defined over
$R$ as well. The formal series $\sum_{i= 1}^\infty a_i \phi_i(x)$ is said to be an asymptotic expansion (series) to $n$ terms of 
$f(x)$ as $x \rightarrow x_0$ if 
\begin{equation}
f(x) = \sum_{i= 1}^n a_i \phi_i(x) + o(\phi_{n+1}(x)) \,\,\, \text{as} \,\,\, x\rightarrow x_0 \>.
\end{equation} 
The formal series $\sum a_i \phi_i$ will be called an asymptotic series. An asymptotic series can be divergent or convergent itself as $n\rightarrow \infty$. For more details see \textsl{e.g.}~\cite[Chapter 1]{erdelyi}.

\section{Tools}
\bigskip We will use Fubini's theorem, which justifies the interchange of order of iterated integrals of continuous functions, and we will use Watson's Lemma. Loosely speaking, Watson's Lemma allows the interchange of order of summation of a series and of integration even though the radius of convergence of the series is violated (leaving us with a divergent asymptotic series).\\
  \begin{lem}[Watson's Lemma]~\cite{bender1999} and~\cite{copson}
  Assume 
  $\alpha > -1$, $\beta >0$ and $b>0$.  
  If $f(t)$ is a continuous function on $[0,b]$ such that it has asymptotic series expansion 
  \begin{equation}
  f(t) \sim t^{\alpha} \sum_{n=0}^\infty a_nt^{\beta n}, \,\,\, t \rightarrow 0^+ \>,
\end{equation} 
(and if $b=+\infty$ then $f(t) <k\cdot e^{ct}\, (t\rightarrow +\infty)$ for some positive constants $c$ and $k$), then
\begin{equation}
 \int_0^b f(t)e^{-xt}dt \sim \sum_{n=0}^\infty \dfrac{a_n \Gamma (\alpha + \beta n +1)}{x^{\alpha	+\beta n+1}}, \,\,\, x \rightarrow +\infty
\end{equation}
\end{lem}
For a proof of Watson's lemma, see~\cite{bender1999}.\\
We will also use Gauss' formula
\begin{equation}
\dfrac{\Gamma'(z+1)}{\Gamma(z+1)}=\int_{t=0}^{\infty}\dfrac{e^{-t}}{t}-\dfrac{e^{-tz}}{e^t-1}dt \;\;\; \text{for} \;\;\; \Re z>0
\end{equation}
a proof of which can be found for example in~\cite{whittaker}. Alternatively, a more elementary proof can be found in~\cite{sasvari}.\\
The next mathematical tool we need comes from a Laplace transform; using $\xi+\dfrac{1}{2}$ instead of the more common symbol $s$, the Laplace transform of $1$ is 
\begin{equation}
\int_{t=0}^{\infty}e^{-t(\xi+\frac{1}{2})}dt=\dfrac{1}{\xi+\frac{1}{2}}
\end{equation}
by direct integration. The integral converges if $\Re(\xi)>-\frac{1}{2}$. We can then prove the following lemma:\\
\begin{lem}[The logarithm lemma]\label{lemma}
For $\Re z > -1/2,$
\begin{equation}
\ln (z+\frac{1}{2})=\int_{t=0}^{\infty}\dfrac{e^{-t}}{t}-\dfrac{e^{-t(z+\frac{1}{2})}}{t}dt \>.
\end{equation}
\end{lem}
\bigskip
\begin{proof}
Integrate the Laplace transform with respect to $\xi$ from $\xi=\frac{1}{2}$ to $\xi=z$:
\begin{equation}
\int_{\xi=\frac{1}{2}}^{z}\dfrac{d \xi}{\xi+\frac{1}{2}}=\int_{\xi=\frac{1}{2}}^{z}\int_{t=0}^{\infty}e^{-t(\xi+\frac{1}{2})}dtd\xi
\end{equation}
Interchange the order of integration---by Fubini's Theorem this is valid---and since $\int e^{-t(\xi+\frac{1}{2})}d\xi=-\dfrac{e^{-t(\xi+\frac{1}{2})}}{t}$, we have
\begin{equation}
\ln(z+\frac{1}{2})-\ln(\frac{1}{2}+\frac{1}{2})=\int_{t=0}^{\infty}-\dfrac{e^{-t(z+\frac{1}{2})}}{t}+\dfrac{e^{-t(\frac{1}{2}+\frac{1}{2})}}{t}dt
\end{equation}
which proves the lemma.
\end{proof}
\bigskip
\section{The formula like Binet's}
\begin{thm}\label{2.3}
If $z> - \frac{1}{2}$, 
\begin{equation}\label{formula17}
\ln z!=(z+\frac{1}{2})\ln (z+\frac{1}{2})-(z+\frac{1}{2})+\ln\sqrt{2\pi}-\int_{t=0}^{\infty}\dfrac{1}{t}\left( \dfrac{1}{t}-\dfrac{1}{2\sinh \frac{t}{2}}\right) e^{-t(z+\frac{1}{2})}dt
\end{equation}
\end{thm}
\begin{proof}
We start with Gauss' formula and switching to $\Gamma$ notation because the derivative $d\Gamma/dz$ is easily written $\Gamma'$,\\
\begin{equation}
\dfrac{\Gamma'(z+1)}{\Gamma(z+1)}=\int_{t=0}^{\infty}\dfrac{e^{-t}}{t}-\dfrac{e^{-tz}}{e^t-1}dt
\end{equation}
(see \textsl{e.g.}~\cite{whittaker}), and Lemma \ref{lemma}.\\


Rearranging Gauss' formula using $e^{t/2}-e^{-t/2}=2\sinh\frac{t}{2}$, \\
\begin{equation}
\dfrac{\Gamma'(z+1)}{\Gamma(z+1)}=\int_{t=0}^{\infty}\dfrac{e^{-t}}{t}-\dfrac{e^{-t(z+\frac{1}{2})}}{2\sinh\frac{t}{2}}dt
\end{equation}
Subtracting Lemma \ref{lemma},\\
\begin{equation}
\dfrac{\Gamma'(\xi+1)}{\Gamma(\xi+1)}-\ln (\xi+\frac{1}{2})=\int_{t=0}^{\infty}\dfrac{e^{-t(\xi+\frac{1}{2})}}{t}-\dfrac{e^{-t(\xi+\frac{1}{2})}}{2\sinh\frac{t}{2}}dt
\end{equation}
Integrating from $\xi=\alpha> -\frac{1}{2}$ to $\xi=z> -\frac{1}{2}$ and interchanging the order of integration using Fubini's theorem, we find (except for a branch issue that we take up later) that
$$\ln \Gamma(z+1)-\ln\Gamma(\alpha+1)-(z+\frac{1}{2})\ln (z+\frac{1}{2})+(z+\frac{1}{2})+(\alpha+\frac{1}{2})\ln(\alpha+\frac{1}{2})-(\alpha+\frac{1}{2})=$$
\begin{equation}
\int_{t=0}^{\infty}\dfrac{1}{t}\left(\dfrac{1}{t}-\dfrac{1}{2\sinh\frac{t}{2}}\right)e^{-t(\alpha+\frac{1}{2})}dt-\int_{t=0}^{\infty}\dfrac{1}{t}\left(\dfrac{1}{t}-\dfrac{1}{2\sinh\frac{t}{2}}\right) e^{-t(z+\frac{1}{2})}dt\>.
\end{equation}
We now need to evaluate the $\alpha$ integral. At $\alpha=0$ Maple and Mathematica can only find a numerical approximation; likewise at $\alpha=\frac{1}{2}$. The numerical approximation can be identified by (for instance) the Inverse Symbolic Calculator at CARMA\footnote{\textbf{https://isc.carma.newcastle.edu.au}. Remark: The ISC is currently down because a security flaw was found. Discussion is under way as to how or if this can be resolved.} (a proof is supplied in Remarks \ref{rmk5.3} and \ref{rmk5.4}.)
\begin{equation}
\int_{t=0}^{\infty}\dfrac{1}{t}\left(\dfrac{1}{t}-\dfrac{1}{2\sinh\frac{t}{2}}\right)e^{-t/2}dt=\frac{1}{2}\ln(\frac{\pi}{e})
\end{equation}
Simplification then yields our formula.\\
\begin{rmk}
According to~\cite{tweddle1984}, this may have been the method Stirling used to identify $\log_{10}\sqrt{2\pi}$, except of course all calculations were done by hand. Apparently, he simply recognized the number $0.39908$. Nowadays very few people could do that unaided, but with the ISC it's easy.
\end{rmk}
\begin{rmk}\label{rmk5.3}
In~\cite{sasvari} we find a trick that could be used to do this integral analytically; we leave this as an exercise.\\
If one desires an actual proof, one can use ``Stirling's formula" (by De Moivre) and leverage the tricky identification of $\sqrt{2\pi}$, as follows.
\end{rmk}
As $z\rightarrow \infty$,
\begin{equation}
\ln \Gamma(z+1)-(z+\frac{1}{2})\ln (z+\frac{1}{2})+(z+\frac{1}{2})\sim \ln\sqrt{2\pi}+\mathcal{O}(\frac{1}{z})\>.
\end{equation}
Therefore (since the second integral goes to $0$ as $z\rightarrow \infty$)
\begin{equation}
\ln\sqrt{2\pi}-\ln \Gamma(\alpha+1)+(\alpha+\frac{1}{2})\ln(\alpha+\frac{1}{2})-(\alpha+\frac{1}{2})
\end{equation}
\begin{equation}
=\int_{t=0}^{\infty}\dfrac{1}{t}\left(\dfrac{1}{t}-\dfrac{1}{2\sinh\frac{t}{2}}\right) e^{-t(\alpha+\frac{1}{2})}dt
\end{equation}
But this is, in fact, our desired theorem with $z=\alpha$.
\end{proof}
\begin{rmk}\label{rmk5.4}
This looks like a circular argument, but it is not. We have here used the $\sqrt{2\pi}$ from the formula popularly known as Stirling's formula, for which there are many proofs analytically (see \textsl{e.g.}~\cite{whittaker}).
\end{rmk}

\begin{cor}~\cite[p.~399]{levinson1970complex}
By analytic continuation, formula \eqref{formula17} holds for $\Re z\geq -1/2$, since the integral is convergent there.
\end{cor}

\section{Evaluation of $\Gamma$ using this divergent series}
\subsection{First attempts}
It has long been known that ``Stirling's approximation'' leads to a viable method to evaluate $\ln\Gamma(z)$. The basic idea is to use the asymptotic series to evaluate $\ln\Gamma(z+n)$ for some large $n$ (large enough that the series gives some accuracy) and then work down with the recursive formula
\begin{equation}
\ln\Gamma(z+n-1)= -\ln(z+n-1)+\ln\Gamma(z+n)
\end{equation}
until we have reached $\ln\Gamma(z)$. This naive idea is surprisingly effective. The point of discussion is just how large $n$ should be, and how many terms in ``Stirling's series'' one should retain, in order to make an effective formula. \\
Given that we now have a different asymptotic formula under consideration (the original, more accurate, but certainly not ``new'' formula) all of the discussion points are necessarily changed. Just as an example, take (say), $z=11+i/2$. If we want $\ln((11+i/2)!)$ then Stirling's original series gives 
\begin{gather*}
\ln\sqrt{2\pi}+(11.5+i/2)\ln(11.5+i/2)-(11.5+i/2)-\dfrac{1}{24(11.5+i/2)}+\mathcal{O}(\dfrac{1}{z^3}) \\
=17.4914469445+1.22148819106i	
\end{gather*}
Wolfram Alpha confirms this, giving
$$\ln((11+i/2)!)\doteq 17.4914485209+1.22148798i \>.$$
Rather than get into the minutiae of how many terms to take, and how far to push the argument to the right, we take a different tack: we look at automatic sequence acceleration of the original divergent series. If 
\begin{equation}
S=\ln\sqrt{2\pi}+Z\ln Z-Z-Z\sum_{n\geq 1}\dfrac{(1-2^{1-2n})B_{2n}}{2n(2n-1)Z^{2n}}\>,
\end{equation}
then we wonder if simple execution of the Maple command 
\begin{equation}
\texttt{evalf}\texttt{(Sum(}\texttt{a(n)}\texttt{,n=1..}\texttt{infinity))};
\end{equation}
where $a(n)$ is defined as $\dfrac{(1-2^{1-2n})B_{2n}}{2n(2n-1)Z^{2n}}$ will automatically produce an accurate result.
\begin{displayquote}
``Sometimes Maple knows things that you don't know. And then you wonder just what.''
\begin{flushright}
--Jon Borwein.
\end{flushright}
\end{displayquote}
\subsection{Levin's $u$-transform}
What Maple knows here is called Levin's $u$-transform. This is a method to accelerate convergence of the sequence of partial sums
\begin{equation}
S_n=\sum_{j=1}^{n}a_j
\end{equation} 
of the series we consider. For an introduction to sequence acceleration, see~\cite{henrici1982} and~\cite{henrici1964}. For an introduction to Levin's $u$-transform, see~\cite{weniger}. \\
\\The basic idea is to replace the sequence $S_0, S_1, S_2,\cdots$ with a new one that has the same limit but which converges faster. More precisely, Levin's $u$-transform for $S_n$ is given as:
\begin{equation}
u_{k}^{(n)}(\beta , S_{n}) = \dfrac{\sum_{j=0}^{k} (-1)^{j} {k \choose j} \dfrac{(\beta + n + j)^{k-2}}{(\beta + n + k)^{k-1}}\dfrac{S_{n+j}}{a_{n+j}}}{\sum_{j=0}^{k} (-1)^{j} {k \choose j} \dfrac{(\beta + n + j)^{k-2}}{(\beta + n + k)^{k-1}}\dfrac{1}{a_{n+j}}}
\end{equation}
The parameter $\beta > 0$ is ``in principle completely arbitrary''~\cite{weniger}. In practice, Maple's routine chooses $\beta=1$.\\
For \textsl{irregular} sequence transforms such as Levin's $u$-transform, this may even transform divergent series into rapidly convergent ones. The price, however, is that it doesn't always work. It works well enough, though, that it is the default method coded in Maple~\cite{geddes1992}. It is accessed most simply by applying the ``evalf" command to an inert sum (denoted by capital-letter \texttt{Sum}). For instance,\\
\begin{equation}
\texttt{evalf}\texttt{(Sum((}\texttt{-2)}^\texttt{n}\texttt{,n=0..}\texttt{infinity))};
\end{equation}
yields $0.3333333333$\footnote{Correctly, in the sense of Euler summation, taking $1+r+r^2+\cdots=1/(1-r)$ even if $|r|>1$ by redefining what the infinite sum actually means: see \textsl{e.g.}~\cite{Hardy}, for more classical work on making sense of divergent series.}.\\
Other sequence acceleration methods or quadratures could be used (see for example chapter $28$ of~\cite{Trefethenbook}), but we wanted to show the capabilities of some (under-appreciated) off-the-shelf~tools.\\
If we issue the command (with a numerical value for $z$, say $z=11+i/2$)
\begin{lstlisting}
> evalf(-(z+1/2)*(Sum((1-2^(1-2*n))*$\mathrm{bernoulli}$(2*n)/
(2*n*(2*n-1)*(z+1/2)^(2*n)), n = 1 .. infinity))+
ln(sqrt(2*Pi))+ln(z+1/2)*(z+1/2)-(z-1/2);
\end{lstlisting}
we get $\ln((11+i/2)!)$ with full accuracy: $14$ digits if Digits $:=15$, $28$ digits if Digits $:=30$, $58$ digits if Digits $:=60$, and so on. This divergent series is being accurately, and quickly, summed by Maple's built-in sequence acceleration using the Levin $u$-transformation method above. \\
If we test this summation by looking at the error
\begin{equation}
\ln\Gamma(z+1)-\ln S(z)
\end{equation}
over a range $-20\leq \Re z\leq 20$, $-20\leq \Im z \leq 20$, we get the curious result in Figure~\ref{relerrRedBlue1}.
\begin{figure}[!h]
    \centering
    \includegraphics[scale=0.4]{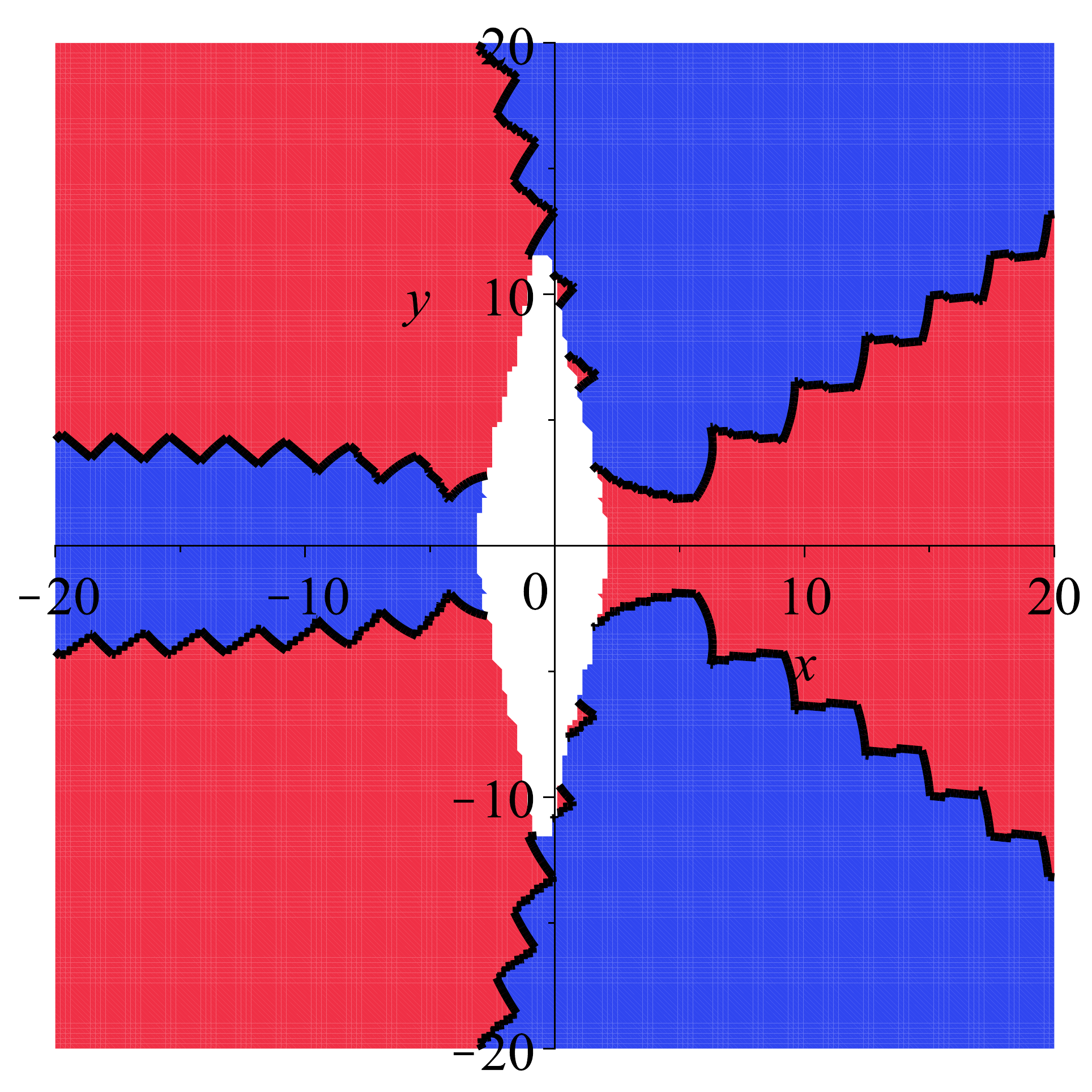}
    \caption{The region of utility for Levin's $u$-transform without an unwinding number.}
    \label{relerrRedBlue1}
\end{figure}
\\
Everywhere in the red region (which includes the real axis for $x$ larger than about $2.1$) has full accuracy, whatever the setting of Digits. The region in white, in the middle, with its scalloped edges, is the region where Levin's $u$-transform fails and Maple returns an unevaluated Sum, as one can see in the example below:

\begin{lstlisting}
> Digits := 20:
\end{lstlisting}
\begin{maplegroup}
\mapleresult
\begin{maplelatex}
\mapleinline{inert}{2d}{Digits := 20}{\[\displaystyle {\it Digits}\, := \,20\]}
\end{maplelatex}
\end{maplegroup}
\begin{lstlisting}
> z := 1+.1*I:
\end{lstlisting}
\begin{maplegroup}
\mapleresult
\begin{maplelatex}
\mapleinline{inert}{2d}{z := 1.+.1*I}{\[\displaystyle z\, := \, 1.0+ 0.1\,i\]}
\end{maplelatex}
\end{maplegroup}
\begin{lstlisting}
> evalf(-(z+1/2)*(Sum((1-2^(1-2*n))*$\mathrm{bernoulli}$(2*n)/
(2*n*(2*n-1)*(z+1/2)^(2*n)), n = 1 .. infinity))+
ln(sqrt(2*Pi))+ln(z+1/2)*(z+1/2)-(z-1/2);
\end{lstlisting}
\begin{maplegroup}
\mapleresult
\begin{maplelatex}
\mapleinline{inert}{2d}{(-1.50000000000000-.1*I)*(Sum((1/2)*(1-2^(1-2*n))*bernoulli(2*n)/(n*(2*n-1)*(1.50000000000000+.1*I)^(2*n)), n = 1 .. infinity))+0.2380532679023624382e-1+0.4062048632794543180e-1*I}
\bigskip
{\[\displaystyle  \left( - 1.50000000000000\\
\mbox{}- 0.1\,i\\
\mbox{} \right) \sum _{n=1}^{\infty }1/2\,{\frac { \left( 1-{2}^{1-2\,n} \right) {\it bernoulli}\left( 2\,n \right) }{n \left( 2\,n-1 \right) \\
\mbox{} \left(  1.50000000000000+ 0.1\,i\\
\mbox{} \right) ^{2\,n}}}\\
\bigskip
\bigskip
+ 0.02380532679023624382+ 0.04062048632794543180\,i\\
\mbox{}\]}
\end{maplelatex}
\end{maplegroup}
%
The boundary of this region is very curious, and we return to the proof of theorem \ref{2.3} to try to understand why. After staring at it for some time, we realize that the transition from 
\begin{equation}
\dfrac{\Gamma'(z+1)}{\Gamma(z+1)} \;\;\;  \text{to} \;\;\; \ln\Gamma(z+1)
\end{equation}
depends on the path that $\Gamma(\xi+1)$ takes as $\xi$ goes from $\xi=1/2$ to $\xi=z$ (a straight line in the $\xi$ variable). But $\Gamma(\frac{1}{2}+t(z-\frac{1}{2}))$ may cross the negative real axis (the branch cut for logarithm) several times as $t$ goes from $0$ to $1$. Writing our answers, as we do, as 
\begin{equation}
\ln z! \sim Z\ln Z-Z+\ln\sqrt{2\pi}+Z\sum_{n\geq 1}\dfrac{(1-2^{1-n})B_{2n}}{2n(2n-1)(Z)^{2n}}
\end{equation}
obscures the fact that the imaginary part of the logarithm on the left is in $(-\pi,\pi]$ while the imaginary part on the right might be anything. To make this equation actually true, we must subtract a multiple of $2\pi i$. To force the imaginary part of $S$ into $(-\pi,\pi]$ there is only one choice: replace $S$ by 
\begin{equation}
S-2\pi i \uwk(S)
\end{equation}
where $\uwk (z)=\left\lceil \dfrac{\Im z-\pi}{2\pi} \right\rceil$ is the unwinding number of $z$ (see~\cite{unwindHighm},~\cite{CorlessUnwind1} and~\cite{CorlessUnwind2}). This means that $\ln z! \sim S-2\pi i\uwk(S)$ not $\sim S$.\\
\begin{rmk}
As pointed out by a referee, this is because the sum $S$ is ``really" asymptotic to the analytic function $\ln\Gamma(z+1)$, obtained by analytic continuation of the function compostion $\ln(\Gamma(z+1))$ for $z>0$. See \textsl{e.g.}~\cite{hare} for details and for some simple formulae for $\uwk(S)$ in special cases.
\end{rmk}
When we plot the error $\ln z! - \ln S + 2\pi i \uwk(S)$ as in Figure~\ref{relerrRedb} we see that whenever the Levin's $u$-transform actually returns an answer, we have only roundoff error. We get essentially perfect accuracy\footnote{Except of course for rounding error. We do not attempt a numerical analysis here, which appears involved. The main difficulty is predicting the number of arithmetic operations.} everywhere to the right of the scalloped boundary in Figure~\ref{relerrRedb}. So far as we know, this result is new. Of course, the detailed accuracy needs a proof: we have only provided experimental evidence, here. What every mathematician wants is a guarantee that the acceleration will work, or a perfect description of just when it will fail. We do not have this. \\
However, when we plot the contours of the error $\ln z! - \ln S + 2\pi i\uwk(S)$ as in Figure~\ref{relerrContour2} we see that the Levin's $u$-transform works as well as could possibly be expected: the visible contours are all less than $10^{-28}$, when we work in $30$ Digits; clearly the error is zero up to roundoff. We have computed the error at ten thousand locations in the region $[0-1000i, 1000+1000i]$ and the maximum error was $10^{-27}$ (on a $100 \times 100$ grid).\\
\begin{figure}[!h]
    \centering
    \includegraphics[scale=0.3]{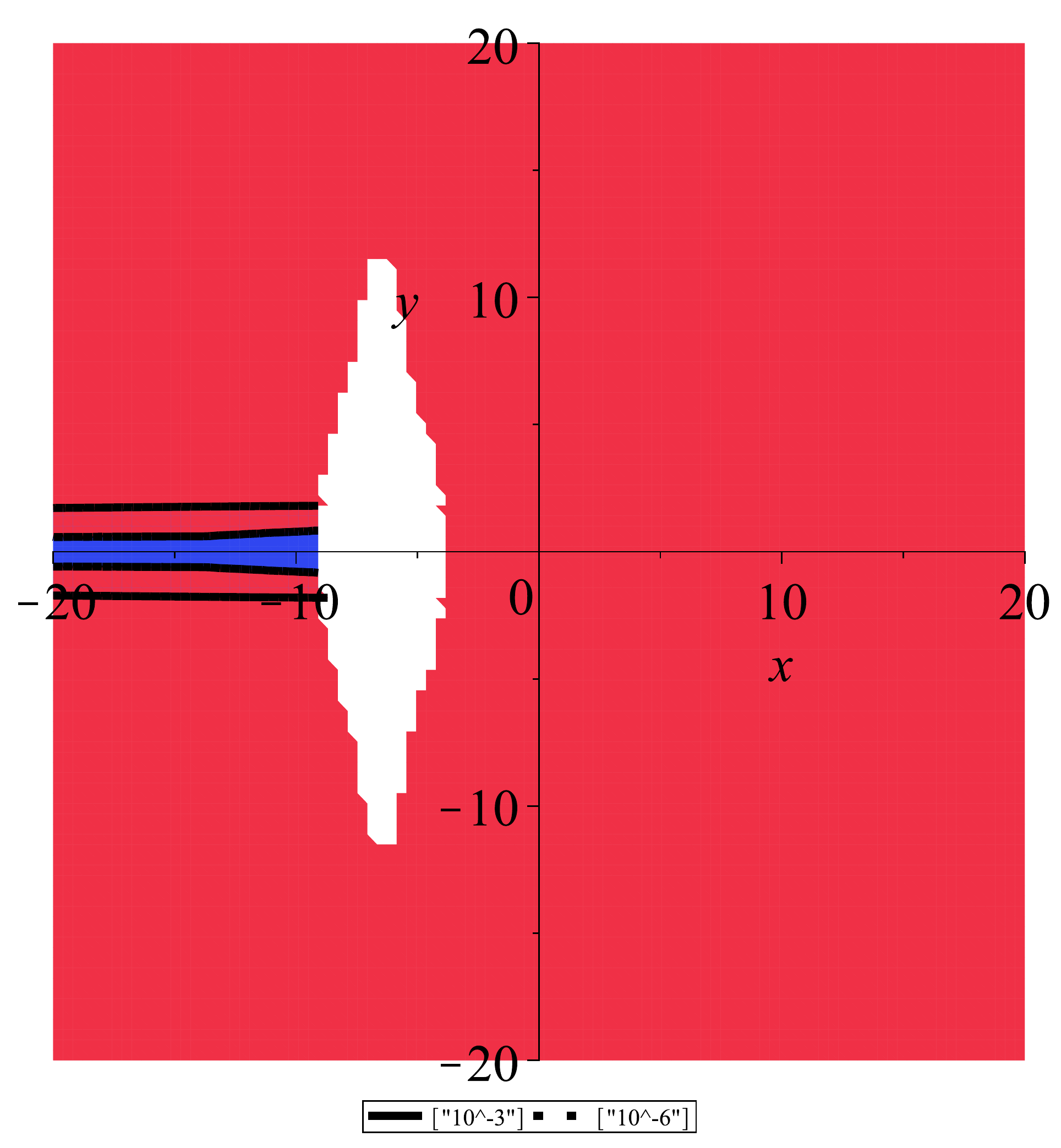}
    \caption{The region of utility for Levin's $u$-transform. We have essentially perfect accuracy (up to roundoff error) outside the region around the negative real axis and the ``lozenge of failure''. Curiously, the error increases gradually near the negative real axis.}
    \label{relerrRedb}
\end{figure}
\begin{figure}[!h]
    \centering
    \includegraphics[scale=0.45]{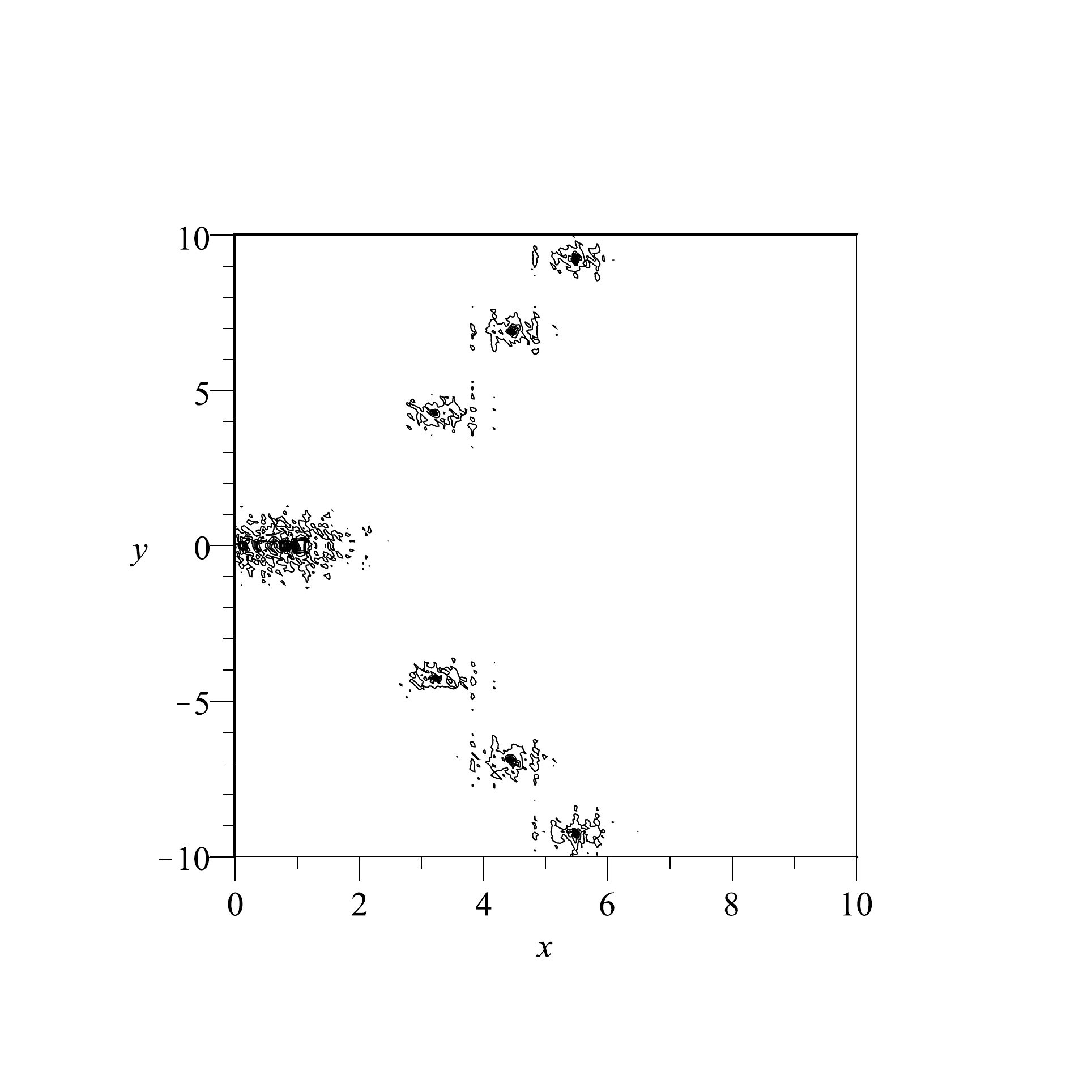}
    \caption{$3D$ plot looking straight down of the error of $\ln z! - \ln S + 2\pi i\uwk(S)$. The errors are everywhere less than $10^{-27}$. We work in $30$ digits of precision.}
    \label{relerrContour2}
\end{figure}
\subsection{Truncating the series without Levin's $u$-transform}
In this section, we plot the absolute estimate error of the truncated series $T$ (not using Levin's $u$-transform) $T-\ln(Z-1/2)!$ where 
\begin{align}
T=u-2\pi i\uwk(u)
\end{align}
and $u=Z\ln(Z)-Z+\ln(\sqrt{2\pi})-\dfrac{1}{24Z}$. For different contours ($10^{-3}$ and $10^{-6}$), we get a very curious result as one can see in Figure~\ref{truncseries}. The error is small outside the keyhole contour. This is more the kind of error we expect from truncated asymptotic series. We see good accuracy even with very few terms. It may be surprising to see that the error is small even in parts of the left half plane, although not near the negative real axis.
\begin{figure}[!h]
    \centering
    \includegraphics[scale=0.45]{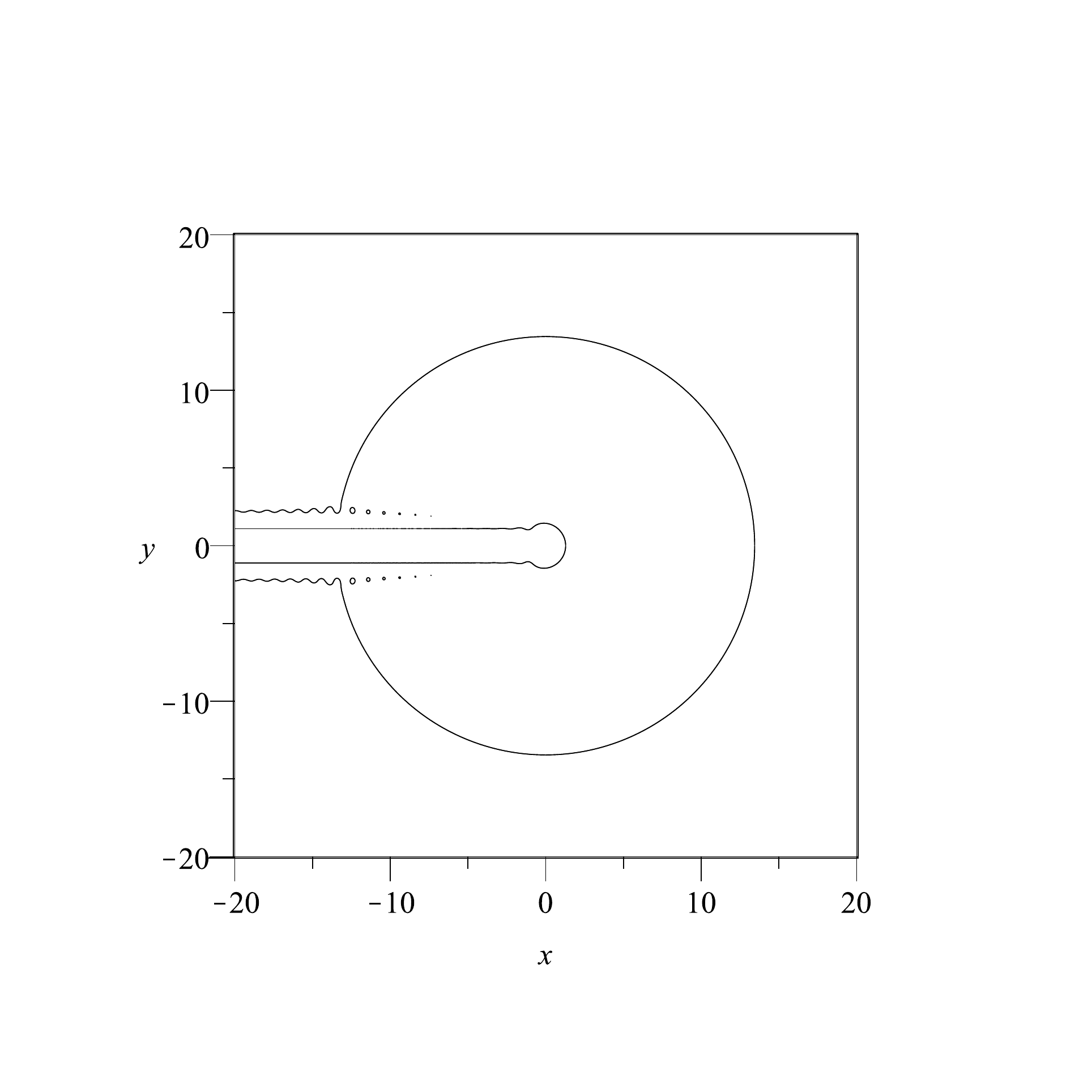}
    \caption{The absolute estimate error of the $T-\ln(Z-1/2)!$. The inner contour is at level $10^{-3}$, and the outer is $10^{-6}$. The truncation error is smaller outside each contour. We used Digits $=30$ and grid $=[600,600]$ in the construction of this figure. The ``bubbles'' and ``wiggles'' in this figure are unexplained.}
    \label{truncseries}
\end{figure}
\section{Concluding Remarks}
The Gamma function and the factorial function, invented in the $1700$'s, have been very thoroughly studied. Richard Brent's article~\cite{brentarxiv} points out some facts, known to Hermite and to Gauss, that were not covered in the survey~\cite{robgamma}, which looked at about $100$ references. One learns therefore that it is difficult to claim a result (formula or proof) is truly new; we are worried in particular that Gauss knew of our Binet--like formula proved here.\\
Nonetheless we believe the proof and numerical experiments have some value in the modern literature. The appearance of the unwinding number in the asymptotic series (either Stirling's or De Moivre's) may also be of value for people who write programs to compute $z!$.
\bibliographystyle{line}
\bibliography{bib}

\end{document}